\documentclass[12pt,a4paper]{amsart}
\usepackage[utf8]{inputenc}

\usepackage{graphicx,color}
\usepackage{amsmath}
\usepackage{amssymb}
\usepackage{amsfonts}
\usepackage{amsthm}
\usepackage{amscd}
\usepackage{ae}
\usepackage[T1]{fontenc}

\newcommand{\D}{\mathbb{D}}

\font\sets=msbm10 scaled \magstep1
\def\R{\text{\sets R}}

\def\N{\text{\sets N}}

\def\C{\text{\sets C}}

\newcommand{\sss}{\mathcal{S}}

\renewcommand\Re{\operatorname{Re}}

\theoremstyle{plain}
\newtheorem{theorem}{Theorem}
\newtheorem{proposition}[theorem]{Proposition}
\newtheorem{lemma}[theorem]{Lemma}

\newtheorem{conjecture}[theorem]{Conjecture}
\author[O. Hirviniemi]{Olli Hirviniemi}
\address{University of Helsinki, Department of Mathematics and Statistics, P.O. Box 68, FIN-00014 University of Helsinki, Finland}
\email{olli.hirviniemi@helsinki.fi}

\author[I. Prause]{Istv\'an Prause}
\address{Department of Physics and Mathematics, University of Eastern Finland, P.O. Box 111, 80101 Joensuu, Finland}
\email{istvan.prause@uef.fi}

\author[E. Saksman]{Eero Saksman}
\address{University of Helsinki, Department of Mathematics and Statistics, P.O. Box 68, FIN-00014 University of Helsinki, Finland}
\email{eero.saksman@helsinki.fi}

\thanks{The work was supported by the Finnish Academy Coe 'Analysis and Dynamics' and the Finnish Academy projects  1266182, 1303765, and 1309940. }

\title[Localized regularity and finite distortion]{Localized regularity of planar maps of finite distortion}

\begin{document}

\begin{abstract}
In this article we study  fine regularity  properties   for mappings of finite distortion. Our main theorems yield  strongly localized regularity results in the borderline 
case in the class of maps of exponentially integrable distortion. Analogues of such results  were known earlier in the case of
 quasiconformal mappings. Moreover, we study regularity for  maps whose distortion has better than exponential integrability.

\smallskip



\end{abstract}

\maketitle

\section{Introduction}\label{se:intro}
Let $f : \Omega \to \C$ be a function where $\Omega\subset \C$ is a domain. We say that $f$ is a (homeomorphic and orientation preserving) mapping of finite distortion if following conditions are satisfied.

\smallskip

(i) $f \in W^{1,1}_{loc}(\Omega)$.

\smallskip

(ii) $f:\Omega\to f(\Omega)$ is a homeomorphism with $J_f\geq 0$ a.e.

\smallskip

(iii) $|Df|^2 = K_f(z) J(z,f)$ for a.e. $z \in \C$, where $K_f$ is a measurable function that is finite almost everywhere. 

\smallskip

In an analogous way one may define mappings of finite distortion on subdomains of $\R^d$. In this article we only consider mappings of finite distortion on the plane. 
A planar mapping of finite distortion satisfies a Beltrami equation
\[
\overline{\partial}f(z) = \mu_f(z) \partial f(z),
\]
where $\mu_f$ is a measurable function with $|\mu_f(z)|<1$ for a.e. $z$. One has $|\mu_f(z)|=\frac{K_f(z)-1}{K_f(z)+1}$. Here and henceforth we employ the standard notation $\overline{\partial}:=\frac{d}{d\overline{z}}=\frac{1}{2}(\partial_x+i\partial_y)$ and $\partial:=\frac{d}{dz}=\frac{1}{2}(\partial_x-i\partial_y)$.

An important subclass is formed by mappings of finite exponential distortion which have the property that for some positive constant $p>0$ one has
\begin{equation}\label{eq:fed}
e^{pK_f(z)}\in L^1_{loc}.
\end{equation}
A natural version of the measurable Riemann mapping theorem, Stoilow  factorization theorem, and many other basic features of the standard quasiconformal theory generalise to these classes.
 For a good account of the theory we refer the reader to \cite[Chapter 20]{AIM}. Improving on  earlier results \cite{FKZ} (which has result valid also in $n$ dimensions), it was shown in \cite{AGRS} that for a mapping of exponentially integrable distortion satisfying \eqref{eq:fed} there is the regularity
\begin{equation}\label{eq:fed2}
|Df|^2\log^{\beta}(e+|Df|)\in L^1_{loc}\qquad \textrm{for}\quad \beta<p-1,
\end{equation}
and this is not necessarily true for $\beta=p-1.$ Note that in the above results   one is  interested only in the local regularity of mappings of finite distortion with exponentially integrable distortion.  Similarly, in our work it is enough to consider only the regularity of principal maps near the origin since local regularity results may then be transferred by the Stoilow factorisation theorem to maps that are defined  on subdomains.  Let us recall that a principal map $f:\C\to\C$  is conformal (i.e.  $\mu_f(z) = 0$) outside the unit disk $\D$
and  $f(z) = z + O(1/|z|)$ near infinity. 

For (standard) $K$-quasiconformal maps (i.e. $K_f(z)\leq K<\infty$, where $K$  is a constant), the optimal area distortion \cite{As} implies that $Df\in L_{loc}^p$ for $p<\frac{2K}{K-1}$,  but this fails in general in the borderline case $p=\frac{2K}{K-1}$. There is a substitute \cite[Cor. 13.2.5]{AIM} in the form of inclusion in the weak space $Df\in L_{loc}^{\frac{2K}{K-1},\infty}$. Another kind of  result in the borderline case was given in \cite[Theorem 3.5]{AIPS1} stating that a $K$-quasiconformal map satisfies
\begin{equation}\label{eq:qc1}
(K-K_f(z))|Df|^{\frac{2K}{K-1}}\in L^1_{loc},
\end{equation}
which gives \emph{strongly localized regularity} information on the map, especially $\int_{K_f\leq K-\varepsilon}|Df|^{\frac{2K}{K-1}} <\infty$ for all $\varepsilon >0$. For further basic results on planar maps of exponentially integrable distortion we refer e.g. to \cite{BrJe,RSY, IKM, K, GKT} and the references therein.

The principal aim of the present note is to establish a strongly localized regularity result  for  mappings of finite distortion analogous to \eqref{eq:qc1}. Our main result states the following:
\begin{theorem}\label{thm:main1}
Assume that $f$ is a planar (homeomorphic) mapping of exponentially integrable distortion satisfying \eqref{eq:fed} with $p=1$.  Then it holds that
\begin{equation}\label{eq:mfd3}
\int_A \frac{1}{\log^{4+\varepsilon}\left(e+K_f\right)} |D f|^2 < \infty \qquad \textrm{for any}\quad \varepsilon>0
\end{equation}
for any compact subset $A\subset \Omega$.
\end{theorem} 

Arguments in the proof of Theorem \ref{thm:main1} also give the following result. 

\begin{theorem}\label{co:3}
With $f$ a mapping of integrable distortion satisfying  \eqref{eq:fed} for some $p>0$ we have
$$
|Df|^2\log^{p-1}(e+|Df|)[\log\log(|Df|+10)]^{-(1+3p+\varepsilon)}\in L^1_{loc}\qquad \textrm{for any}\quad \varepsilon>0. 
$$
\end{theorem}

In the radial case, one can  improve both Theorems \ref{thm:main1} and \ref{co:3}: 
\begin{theorem}\label{th:radial}
{\rm (i)}\quad Let $f: \C \to \C$ be a planar and radial homeomorphic mapping of finite distortion with exponentially integrable distortion satisfying \eqref{eq:fed} with $p=1$. 
Then
\[
\int_{A}\frac{1}{\log^{1+\varepsilon}\left(e+K_f\right)}  |D f|^2 < \infty \qquad \textrm{for any}\quad \varepsilon>0
\]
and any compact subset $A\subset \C$.

\smallskip
\noindent {\rm (ii)}\quad For radial maps of $p$-integrable distortion with  $p>0$ we have

$$
|Df|^2\log(e+|Df|)^{p-1}\log\log(|Df|+10)^{-(1+\varepsilon)}\in L^1_{loc}\qquad \textrm{for any}\quad \varepsilon>0.
$$
\end{theorem}

Our next result  considers mapping that in a sense lie in between mappings of exponentially integrable distortion and standard quasiconformal maps.
\begin{theorem}\label{thm:main2}
Assume that $f$ is a planar homeomorphic mapping of finite distortion satisfying the integrability
\begin{equation}\label{eq:fed4}
e^{(K_f(z))^\alpha}\in L^1_{loc}\nonumber
\end{equation} 
for some $\alpha >1$. Then it holds that
\begin{equation}\label{eq:mfd5}
\int_\mathbb{D} |D f|^2\exp\big(\log^\beta(e+|D f|)\big) < \infty
\end{equation}
for any $\beta<1-1/\alpha$. The result is optimal in the sense that the conclusion fails for  $\beta >1-1/\alpha$.
\end{theorem}
\noindent Sharpness of the previous theorem is shown by the map
$$
f(z)= c_\alpha \frac{z}{|z|}\exp\big(-{\textstyle\frac{2}{1-1/\alpha}}\log^{1-1/\alpha}(e+1/|z|)\big) 
$$
for $|z|<1$, and identity outside $\D$. We expect that the extremal maps for Theorems   {\rm \ref{thm:main1}} and  {\rm\ref{co:3}} are  also  given by radial maps,
so it is natural to state:
\begin{conjecture}\label{con:1} The conclusions of Theorem {\rm \ref{th:radial}} remain true without assuming that the map is radial.
\end{conjecture}

Theorem \ref{th:radial} is sharp up to the possible borderline case. For any $0 < \varepsilon < 1$, we can choose $g_\varepsilon : \C \to \C$ to be
\[
g_\varepsilon(z) := \frac{z}{|z|} \left[ \log \left( e + \frac{1}{|z|} \right) \right]^{-p/2} \left[ \log \log \left( e + \frac{1}{|z|} \right) \right]^{- \varepsilon / 2} \qquad \textrm{for}\quad |z| < 1 
\]
and $g_{\varepsilon} (z) := cz$ elsewhere for some constant $c$. Then one directly verifies that $g_\varepsilon$ is a (radial) mapping of finite distortion satisfying \eqref{eq:fed} with $p$ but we have for general $p$
\[
\int_\D |D(g_\varepsilon)|^2\log(e+|D(g_{\varepsilon})|)^{p-1}\log\log(|D(g_\varepsilon)|+10)^{-1+\varepsilon} = \infty
\]
and for $p=1$ we have 
\[
\int _\D \frac{1}{\log^{1-\varepsilon}\left(e+K_{g_\varepsilon}\right)}  |D(g_\varepsilon)|^2 = \infty.
\]

Section \ref{se:mainproof} below contains the proof of Theorems \ref{thm:main1} and \ref{co:3} assuming the quantitative estimate of Lemma \ref{le:kaksi}.  Next, Section \ref{se:neumann} gives careful quantitative estimates for the decay of the Neumann series associated with the Beltrami equation. Then in Section \ref{sec:proofoflemma2} we are ready to accomplish the proof of Lemma \ref{le:kaksi}, and also to complete the proof of Theorem \ref{thm:main2}. Finally,  Section \ref{se:radialproof} treats the case of radial mappings, i.e. Theorem \ref{th:radial}.

\bigskip

\noindent{\bf Acknowledgements:}\quad We are grateful for the anonymous referees for careful reading of the paper and their thoughtful comments that led to substantial improvements in the presentation of the paper.

\section{Proof of  Theorems \ref{thm:main1} and  \ref{co:3}}\label{se:mainproof}
In this section we prove  Theorem \ref{thm:main1} as well as Theorem \ref{co:3} assuming Lemma \ref{le:kaksi} whose proof we provide later. It is useful to note the general comparison for mappings of finite distortion $|\partial f| \leq |Df| \leq 2 |\partial f |$.

\begin{proof}[Proof of Theorem \ref{thm:main1}]
Our basic assumption is  that $f$ is  a principal mapping of finite distortion with 
\begin{equation}\label{eq:-1}
\int_\D e^{K_f}\leq \widetilde C <\infty,
\end{equation}
and we denote by $\mu$ the Beltrami coefficient of $f$.
However, we first consider the class   of quasiconformal $f$ that satisfy \eqref{eq:-1} with a fixed $\widetilde C$. After obtaining uniform estimates for this class, we then at the end of the proof use approximation to deduce results for maps of finite distortion.

We next fix $0<\varepsilon < 1/2$ and  for any $w$ with $0 \leq \Re w \leq 1$ we let $f_w$ be the unique principal solution to the Beltrami equation
\[
\overline{\partial}f_w(z) = \nu_w(z) \partial f_w (z),
\]
where 
\[
\nu_w(z) := \frac{\mu(z)}{|\mu(z)|}|\mu(z)|^{w+\varepsilon}.
\]
A main idea in the proof is to consider the functions
$$
g_w = (1-|\mu|)^{(1-w)/2}\partial f_w
$$
 and apply the analytic interpolation theorem, or actually a very special case of it that reduces to a vector-valued Phragm\'en-Lindel\"of type maximum principle.  

To accomplish this, note that since the dependence $w\mapsto \nu_w$ is analytic, we deduce by the Ahlfors-Bers theorem that the dependence $w\mapsto f_w$ (as an $L^2(\D)$-valued function) is analytic over the closed strip, and hence also $g$ depends analytically on $w$. Especially,  the map $w \mapsto g_w$ is continuous in the strip and analytic in the interior. Moreover, by a standard application of the Neumann series and the definition of $g$ we see that $\|g_w\|_{L^2(\D)} \leq C(\widetilde C)$  for all $w$ in the closed strip $\{0 \leq \Re w \leq 1\}$. Fix $h \in C_0^\infty(\D)$ with $\|h\|_{L^2(\D)} = 1$. A fortiori, the function $w \mapsto \int_{\D} g_w(z)h(z) \, dm(z)$ is a continuous and bounded analytic function  in the closed strip and analytic in the interior. If we denote 
\[
\widetilde M_r := \sup_{\Re w = r} \left|\int_{\D} g_w(z)h(z) \, dm(z) \right|,
\]
and 
\[
 M_r := \sup_{\Re w = r} \|g_w\|_{L^2(\D)},
\]
then we have by a classical version of the Hadamard's three lines theorem that
\begin{equation}\label{eq:apu0}
\widetilde M_\theta \leq \widetilde M_0^{1-\theta} \widetilde M_1^\theta \leq   M_0^{1-\theta} M_1^\theta\nonumber .
\end{equation} 
Since $h \in C_0^\infty(\D)$ is arbitrary we in fact have for any $\theta\in (0,1)$
\begin{equation}\label{eq:apu1}
 M_\theta \leq  M_0^{1-\theta} M_1^\theta.
\end{equation} 

In order to continue the proof we need several auxiliary results.
\begin{lemma}\label{le:yksi}
For any $w$ with $\Re w = 0$ we have
$
\displaystyle \int_{\D} |g_w|^2 \leq {C}{\varepsilon}^{-1}
$
with a universal constant $C$. In particular, $M_0 \leq {C_0}{\varepsilon^{-1/2}}$.
\end{lemma}
\begin{proof}
As $f_w$ is a quasiconformal principal mapping, we obtain by the Bieberbach area formula
\[
\int_{\D} J(z,f_w) = |f(\D)| \leq \pi.
\]
To use this, note first that as $J(z,h) = |\partial h|^2 - |\overline{\partial}h|^2$, we have by the definition of $g$ for any $w$ with $\Re w=0$
\[
\int_{\D} |g_w|^2 = \int_{\D} (1-|\mu|)|\partial f_w|^2 = \int_{\D} \frac{1-|\mu|}{1-|\nu_w|^2} J(z,f_w)
\leq \int_{\D} \frac{1-|\mu|}{1-|\mu|^{2\varepsilon}} J(z,f_w).
\]
As $x \mapsto x^{2\varepsilon}$ is a concave function whose derivative at $x=1$ equals $2\varepsilon$, we have $x^{2\varepsilon} \leq 1 + 2\varepsilon(x-1)$ for all $x>0$. This implies that 
\[
\frac{1-|\mu|}{1-|\mu|^{2\varepsilon}} \leq \frac{1-|\mu|}{2\varepsilon(1-|\mu|)} = \frac{1}{2\varepsilon},
\]
finishing the proof.
\end{proof}
\begin{lemma}\label{le:kaksi}  
For any $w$ with $\Re w = 1$ it holds that
$
\displaystyle \int_{\D} |g_w|^2 \leq {C}{\varepsilon}^{-4}.
$
The constant $C$ depends only on $\widetilde C$ in \eqref{eq:-1} . In particular, $M_1 \leq {C_0}{\varepsilon^{-2}}$.
\end{lemma}
\noindent We postpone  the proof of this lemma to Section \ref{sec:proofoflemma2} as it needs  more preparation, especially one needs to carefully check the dependence of constants in certain arguments of \cite{AGRS}.

In order to continue the proof we choose $\theta = 1- \varepsilon$ in \eqref{eq:apu1} and note that $f_{1-\varepsilon}=f $ in order to obtain 
\begin{equation}\label{eq:jep}
\int_{\D} (1-|\mu(z)|)^{\varepsilon} |\partial f |^2 \, dm(z)  \;=\; \int_{\D} (1-|\mu(z)|)^{\varepsilon} |\partial f_{1-\varepsilon}(z)|^2 \, dm(z)\; \leq \;\frac{C}{\varepsilon^4}.
\end{equation}

Up to now we have considered the case where  $f$ is quasiconformal and satisfies  \eqref{eq:-1}. We then choose a sequence of quasiconformal maps that converge to $f$ locally uniformly and satisfy the condition \eqref{eq:-1}. In order to find such a sequence  one may e.g. use the factorization (see \cite[Corollary 4.4]{AGRS}) 
$f=g\circ h$, where $g$ is  (say) 5-quasiconformal and $\int_\D e^{5K_h}\leq \widetilde C+\pi e^5$. In this situation one may approximate $h$ by quasiconformal maps $h_n$ by truncating its dilatation in a standard way and one defines $f_n:=g\circ h_n$, and then the sequence $f_n$ satisfies \eqref{eq:-1} with possibly slightly increased $\widetilde C$, but uniformly in $n$. That we have the convergence $h_n\to h$ locally uniformly is deduced by the fact that in this regime the Neumann-series of $h_n$ converges in $L^2$ with uniform bounds for $k$:th term, and clearly we have convergence in $L^2$ for each individual term of the Neumann-series.  Thus $Dh_n\to Dh$ in $L^2_{loc}$, which implies local convergence in $VMO$ for the maps $h_n$, and  the uniform convergence then follows by the known uniform modulus of continuity estimates for mappings of exponentially integrable distortion.

Since $f_n\to f$ locally uniformly, we obtain that $f_n\to f$ in the sense of distributions, and hence $\partial f_n\to \partial f$ in the sense of distributions.  Let us then fix $p<2$. We  have by \eqref{eq:fed2} that $\int_{B} |\partial f_n|^p\leq C$ uniformly in $n$ for any fixed ball, and the same inequality holds also for $f$ instead of $f_n$. This verifies (by using the density of test functions in $L^q$) that the convergence  in distributions upgrades to weak convergence $\partial f_n\overset{w}{\to} \partial f$ in $L^p(\D).$ This immediately implies the weak convergence in $L^p(\D)$ of $(1-|\mu(z)|)^{\varepsilon/p}\partial f_n$ to $(1-|\mu(z)|)^{\varepsilon/p}\partial f$, and we obtain by the basic properties of weak $L^p$-convergence and the uniform estimate \eqref{eq:jep} that
\begin{eqnarray*} 
&&\int_{\D} (1-|\mu(z)|)^{\varepsilon} |\partial f |^p \, dm(z) \\ &\leq& \liminf_{n\to\infty}\int_{\D} (1-|\mu(z)|)^{\varepsilon} |\partial f_n |^p \, dm(z)
\; \leq\;  \liminf_{n\to\infty}\int_{\D} (1-|\mu(z)|)^{\varepsilon} |\partial f_n |^2 \, dm(z) \; +\pi \\
&\leq&\frac{C'}{\varepsilon^4}.
\end{eqnarray*}
By letting $p\nearrow 2$ we finally obtain for the general $f$ the desired inequality
\begin{equation}\label{eq:apu5}
\int_{\D} (1-|\mu(z)|)^{\varepsilon} |\partial f|^2 \leq \frac{C'}{\varepsilon^4},
\end{equation}
and again, the  constant $C'$ in \eqref{eq:apu5} does not depend on $\varepsilon$. 

The inequality \eqref{eq:apu5} already provides a non-trivial localization result because we may consider small values of $\varepsilon$. However, as we have all values $\varepsilon\in(0,1/2)$ at our disposal, the result can be improved on by invoking the following observation:
\begin{lemma}\label{le:apu3}
Let $h$ and $W$ be non-negative functions on $\D$, with $W(z) \leq 1$ for all $z$. Let also  $\varepsilon_0\in (0,1/2)$, $\alpha,C>0$ be positive constants. Assume that for  any $0<\varepsilon<\varepsilon_0$ we have
\begin{equation}\label{eq:apu10}
\int_{\D} (W(z))^{\varepsilon} h(z) \, dm(z) \leq \frac{C}{\varepsilon^\alpha}.
\end{equation}
Then there is a constant $C_1=C_1(\varepsilon_0, \alpha, C)$ such that for $0<\eta\leq 1$
\[
\int_{\D} \frac{1}{\left(\log\left(e + \frac{1}{W(z)}\right)\right)^{\alpha+\eta}} h(z) \, dm(z) \leq \frac{C_1}{\eta}.
\]
\end{lemma}
\begin{proof}
The  assumption remains valid if $W$ is replaced by $\min(W, 1/2)$ and the conclusion obtained in this case yields the original one, in view of the assumption. We may hence assume that $W(z)\in [0,1/2]$ for all $z$. From \eqref{eq:apu10} it immediately follows that if $0< \eta \leq 1$, then
\[
\int_0^{\varepsilon_0} \varepsilon^{\alpha-1+\eta} \int_{\D} (W(z))^{\varepsilon}h(z) \, dm(z)\, d\varepsilon \leq C \int_0^{\varepsilon_0}  \varepsilon^{\eta-1} \, d\varepsilon = \frac{C}{\eta} \varepsilon_0^{\eta}.
\]
On the other hand, we can use Fubini's theorem to conclude that
\[
\int_0^{\varepsilon_0}  \varepsilon^{\alpha-1+\eta} \int_{\D} (W(z))^{\varepsilon}h(z) \, dm(z)\, d\varepsilon = \int_{\D} h(z) \int_0^{\varepsilon_0}  \varepsilon^{\alpha-1+\eta} (W(z))^{\varepsilon} d\varepsilon \, dm(z).
\]
For those $z$ with $W(z)=0$, the inner integral is $0$. Let now $0<a:=W(z)\leq 1/2$.  Then the inner integral is equal to
\begin{eqnarray*}
&&\int_0^{\varepsilon_0}  x^{\alpha+\eta-1} a^{x} \, dx = \frac{1}{\left(\log\left(\frac{1}{a}\right)\right)^{\alpha+\eta-1}} \int_0^{\varepsilon_0} \left(\log\left(\frac{1}{a}\right)x\right)^{\alpha+\eta-1} e^{-\left(\log\left(\frac{1}{a}\right)x\right)} \, dx \\&=& \frac{1}{\left(\log\left(\frac{1}{a}\right)\right)^{\alpha+\eta}} \int_0^{\varepsilon_0\log\left(\frac{1}{a}\right)} s^{\alpha+\eta-1} e^{-s} \, ds.
\end{eqnarray*}
The  last integral factor approaches $\Gamma(\alpha+\eta)$ uniformly on $\eta\in [0,1]$ as $a \to 0$.  The positive function $\phi: (0,1/2] \times [0,1] \to \R$, 
\[
\phi(a,\eta) := \frac{\left(\log\left(e + \frac{1}{a}\right)\right)^{\alpha+\eta}}{\left(\log\left(\frac{1}{a}\right)\right)^{\alpha+\eta}} \int_0^{\log\left(\frac{1}{a}\right)\varepsilon_0} s^{\alpha+\eta-1} e^{-s} \, ds,
\]
extends therefore to a continuous positive function on $[0,1/2]\times [0,1]$. 
Therefore there is a positive constant $c>0$ so that for all $z$ we have
\[
\int_0^{\varepsilon_0} \varepsilon^{\alpha-1+\eta} (W(z))^{\varepsilon} d\varepsilon \geq \frac{c}{\left(\log\left(e + \frac{1}{W(z)}\right)\right)^{\alpha+\eta}}
\]
which finishes the proof.

\end{proof}
Theorem \ref{thm:main1} is  obtained by applying Lemma \ref{le:apu3} in conjunction with inequality \eqref{eq:apu5} using the choices $W(z):=(1-|\mu(z)|)$ and $\alpha=4.$
\end{proof}

\noindent \textbf{Remark.} Generalizing Theorem \ref{thm:main1} for values $p \neq 1$ appears to require interpolating in Orlicz space settings instead of $L^2$ with suitable counterparts of Lemma \ref{le:yksi} and Lemma \ref{le:kaksi}. We have not attempted to carry out the necessary details for the generalisations  since it would considerably increase  the technicality of the paper.

\begin{proof}[Proof of Theorem \ref{co:3}] Following the argument of the proof of Lemma \ref{le:kaksi} and keeping track of the dependence of constant factors, we obtain under the assumption  \eqref{eq:fed} that instead of the result stated in Lemma \ref{le:kaksi} we obtain for general $p$ that
$$
\int_\D |Df|^2\log(e+|Df|)^{p-1}\log^{-\varepsilon}\big(e+|Df|\big) \leq C\varepsilon^{-(1+3p)}.
$$
Then, as before the statement follows by an application of Lemma \ref{le:apu3}.
\end{proof}

\textbf{Remark.}
We note that one may apply Lemma \ref{le:apu3} again directly on the result stated in Theorem \ref{thm:main1} and this yields that the integral 
\begin{equation*}
\int_A \frac{1}{\log^{4+\varepsilon}\left(e+K_f\right)} |D f|^2
\end{equation*}
is bounded by $\frac{C}{\varepsilon}$. Thus  taking $\alpha=1$ in Lemma \ref{le:apu3} we obtain a statement of the form
\begin{equation}\label{eq:mfd4}
\int_\mathbb{A} \frac{1}{\log^{4}\left(e+K_f\right)(\log\log\left(10+K_f\right))^{1+\varepsilon}} |D f|^2 < \infty \qquad \textrm{for any}\quad \varepsilon>0\nonumber.
\end{equation}
An industrious reader may refine this result by iterating the lemma, obtaining estimates for weights with more iterations of logarithms.

\section{Decay of the Neumann series}\label{se:neumann}

For the proof of Lemma \ref{le:kaksi} we need quantitative versions of several auxiliary results in \cite{AGRS}. In this section we  establish decay estimates for the Neumann series that  suffice both for Theorem \ref{thm:main1} and for Theorem \ref{thm:main2}. Our proof follows rather closely the ideas of \cite{Da,AGRS} but keeping track of the dependence of the constants is somewhat non-trivial even in the case $\alpha=1$ which relates to that considered in \cite{Da,AGRS}.

The Beurling operator $\mathcal{S}$ is the  singular integral
\[
\sss \phi(z) := -\frac{1}{\pi} \int_{\C} \frac{\phi(\tau)}{(z-\tau)^2}d\tau .
\]
Recall that in the context of quasiconformal mappings, the classical Beltrami equation in $W^{1,2}_{loc} (\C)$
\[
\overline{\partial} f (z) = \mu(z) \partial f(z)
\]
has a unique principal solution $f(z) = z + O(1/z)$ -- for this and other basic facts on quasiconformal maps we refer the reader to \cite{Ah,AIM}. We can use the identity $\partial f - 1 =  \sss ( \overline{\partial} f)$ to write the Beltrami equation equivalently for $\omega = \overline{\partial} f$
\[
\omega(z) = \mu(z)(\sss\omega(z) + 1),
\]
which is solved by the Neumann series
\[
\omega = (\mathbb{I} - \mu\sss)^{-1} \mu = \mu + \mu \sss \mu + \mu \sss \mu \sss \mu + \cdots.
\]
The series converges absolutely when $|\mu(z)| \leq k < 1$ almost everywhere because $\sss$ is a unitary operator in $L^2(\C)$. This is no longer true if only $|\mu(z)|<1$, but we have as substitute the estimates of Lemma \ref{le:Neumann}.  We state here a refined (in the case $\alpha=1$) and generalized  (for $\alpha>1$)  version of \cite[Theorem 3.1]{AGRS} needed for our purposes. Its proof is adapted from the original proof in \cite{AGRS}.

\begin{lemma}\label{le:Neumann} Let $|\mu(z)| < 1$ almost everywhere, with $\mu(z) \equiv 0$ for $|z|>1$. Assume that the  distortion function $K(z) = \frac{1+|\mu(z)|}{1-|\mu(z)|}$ satisfies
\[
e^{K^\alpha} \in L^p(\D),  \quad\textrm{for some}\quad p > 0 \;\;\textrm{and}\;\;\alpha \geq 1
\]
In case $\alpha >1$ we have for every $p>0$ and $\beta\in [p/2,p)$
\[
\int_{\C} |(\mu \sss)^n \mu|^2 \leq C \exp\left( -2(\beta/2)^{1/\alpha}\frac{1}{1-1/\alpha}\big((n+\beta/4+1)^{1-1/\alpha}-(\beta/4+1)^{1-1/\alpha}\big)\right) , \quad n \in \N,
\]
where by denoting $\delta:= \frac{(p-\beta)^2}{\beta(p+\beta)}$, \;  $ \widetilde C:=\frac{8p}{p-\beta}\left( \int_\D e^{pK^\alpha}\right)^{(p-\beta)/2p},$ $b:=(\beta/2)^{1/\alpha}$, and\\  $B:=\max\Big(\frac{ b}{1-1/\alpha}\left((2b/\delta)^{\alpha-1}-(1+\beta/4)^{1-1/\alpha}\right),0\Big)$ we have
\begin{equation}\label{eq:est1}
  C:=(4\delta^{-2} \widetilde Ce^{2B}+1).
\end{equation}
In the case $\alpha=1$ one instead has
\[
\int_{\C} |(\mu \sss)^n \mu|^2 \leq C_0  \Big(\frac{n+\beta/4+1}{\beta/4+1}\Big)^{-\beta} , \quad n \in \N,
\]
where 
\begin{equation}\label{eq:est2}
C_0:=12^{\beta+3}(p/\beta-1)^{-(5+2\beta)}\left( \int_\D e^{pK}\right)^{\frac{1}{2}(1-\beta/p)}.
\end{equation}
\end{lemma}

\begin{proof}
We first note that a simple computation shows that the case $\alpha=1$ follows from the case $\alpha>1$ by first assuming that $\|\mu\|_\infty <1$ and letting $\alpha\to 1^{+}$  in estimate \eqref{eq:est1}. Hence we may assume that $\alpha >1$ and  start by fixing
$0 < \beta < p$. For $n \in \N$ divide the unit disk into two sets
\[
B_n = \left\lbrace z \in \D : |\mu(z)| > 1 - \frac{2\beta^{1/\alpha}}{(4n)^{1/\alpha} + \beta^{1/\alpha}} \right\rbrace
\]
and
\[
G_n = \D \setminus B_n.
\]
By Chebychev's inequality,
\[
|B_n| \leq \Big(\int_{\D} e^{pK^\alpha}\Big) e^{-4np/\beta}.
\]
The terms of the Neumann series $\psi_n = (\mu \sss)^n \mu$ and the auxiliary terms $g_n$ are obtained inductively
\[
\psi_n = \mu \sss(\psi_{n-1}), \qquad \psi_0 = \mu
\]
and
\[
g_n = \chi_{G_n} \mu \sss(g_{n-1}), \qquad g_0 = \mu.
\]
For $g_n$ we can estimate by using the fact that $\sss$ is $L^2$-isometry to see that
\[
\|g_n\|_{L^2}^2 = \int_{G_n} |\mu S(g_{n-1})|^2 \leq \left( 1 - \frac{2\beta^{1/\alpha}}{(4n)^{1/\alpha} + \beta^{1/\alpha}} \right)^2 \|g_{n-1}\|_{L^2}^2,
\]
and therefore
\[
\|g_n\|_{L^2} \leq \prod_{j=1}^n \left( 1 - \frac{2\beta^{1/\alpha}}{(4j)^{1/\alpha} + \beta^{1/\alpha}} \right)\|g_0\|_{L^2} = \exp\left(\sum_{j=1}^n \log\left( 1 - \frac{2\beta^{1/\alpha}}{(4j)^{1/\alpha} + \beta^{1/\alpha}} \right) \right) \|\mu\|_{L^2}.
\]
As $\log(1-x) \leq -x $ for $x<1$, and  $\|\mu\|_{L^2} \leq \sqrt{\pi}$, it follows that
\begin{eqnarray*}
\|g_n\|_{L^2} &\leq& \exp\left(-2^{1-2/\alpha}\beta^{1/\alpha}\sum_{j=1}^n \frac{1}{j^{1/\alpha} + (\beta/4)^{1/\alpha}}  \right) \pi^{1/2}\\
&\leq&  \exp\left(-2^{-1/\alpha}\beta^{1/\alpha}\sum_{j=1}^n \frac{1}{(j+\beta/4)^{1/\alpha}}  \right) \pi^{1/2},
\end{eqnarray*}
where we applied the inequality $(j)^{1/\alpha} + (\beta/4)^{1/\alpha}\leq 2^{1-1/\alpha}(j+\beta/4)^{1/\alpha}.$
The sum inside the exponential can be estimated by an integral
\begin{align}\label{eq:integral}
\sum_{j=1}^n \frac{1}{(j+\beta/4)^{1/\alpha}} \geq \int_{1}^{n+1} \frac{dx}{(x+\beta/4)^{1/\alpha}  } 
= \frac{(n+1+\beta/4)^{1-1/\alpha}-(1+\beta/4)^{1-1/\alpha}}{(1-1/\alpha)} ,
\end{align}
so that
\begin{equation}\label{eq:aa}
\|g_n\|_{L^2} \leq  \exp \left( -2^{-1/\alpha}\beta^{1/\alpha}  \frac{(n+1+\beta/4)^{1-1/\alpha}-(1+\beta/4)^{1-1/\alpha}}{1-1/\alpha}  \right).
\end{equation}
The difference of $\psi_n$ and $g_n$ is
\[
\psi_n-g_n = \chi_{G_n} \mu \sss(\psi_{n-1}-g_{n-1}) + \chi_{B_n} \mu \sss(\psi_{n-1}).
\]
 For the norms, this gives

\[
\|\psi_n-g_n\|_{L^2} \leq  \left(1- \frac{2\beta^{1/\alpha}}{(4j)^{1/\alpha} + \beta^{1/\alpha}} \right) \|\psi_{n-1}-g_{n-1}\|_{L^2}+ \sqrt{R(n)}
\]
with
\[
R(n) = \|\chi_{B_n} \mu \sss(\psi_{n-1})\|_{L^2}^2 = \int_{B_n} |(\mu \sss)^n\mu|^2. 
\]
By induction and estimating like in \eqref{eq:integral} we deduce
\begin{align}\label{eq:a}
\|\psi_n-g_n\|_{L^2} &\leq  \left(1- \frac{2\beta^{1/\alpha}}{(4n)^{1/\alpha} + \beta^{1/\alpha}} \right) \|\psi_{n-1}-g_{n-1}\|_{L^2}+ \sqrt{R(n)}\nonumber\\
& \leq \sum_{j=1}^n \sqrt{R(j)} \prod_{k=j+1}^n \left(1- \frac{2\beta^{1/\alpha}}{(4k)^{1/\alpha} + \beta^{1/\alpha}} \right)\nonumber\\
&= \sum_{j=1}^n \sqrt{R(j)} \exp\left(\sum_{k=j+1}^n \log \left(1- \frac{2\beta^{1/\alpha}}{(4k)^{1/\alpha} + \beta^{1/\alpha}} \right)\right)\nonumber\\
& \leq \sum_{j=1}^n \sqrt{R(j)} \exp\left(-2^{1-2/\alpha}\beta^{1/\alpha} \sum_{k=j+1}^n \frac{1}{k^{1/\alpha}+(\beta/4)^{1/\alpha}}\right)\nonumber\\
& \leq \sum_{j=1}^n \sqrt{R(j)}\exp\left( -2^{-1/\alpha}\beta^{1/\alpha}  \frac{(n+1+\beta/4)^{1-1/\alpha}-(j+1+\beta/4)^{1-1/\alpha}}{1-1/\alpha}  \right)\nonumber\\
& = \exp\left( -2^{-1/\alpha}\beta^{1/\alpha}\frac{1}{1-1/\alpha}\big((n+1)^{1-1/\alpha}-(1+\beta/4)^{1-1/\alpha}\big)\right) \times
\nonumber
\\& \times\sum_{j=1}^{n} \exp\left(2^{-1/\alpha}\beta^{1/\alpha}\frac{1}{1-1/\alpha}\big((j+1+\beta/4)^{1-1/\alpha}-(1+\beta/4)^{1-1/\alpha}\big)\right) \sqrt{R(j)}.
\end{align}

We next recall that in \cite{AGRS} the Astala area distortion result $|f^\lambda(E)|\leq \pi M|E|^{1/M}$ for quasiconformal maps was used to estimate $R(n)$ by considering  the solution $f = f^\lambda$ to the Beltrami equation
\[
\overline{\partial}f = \lambda \mu \partial f,
\]
with $|\lambda|<1$ via expressing the term $(\mu \mathcal{S})^n \mu$ of the Neumann series
$
\overline{\partial}f^\lambda = \sum_{n=0}^\infty \lambda^{n+1} (\mu \mathcal{S})^n \mu.
$
by a Cauchy integral
\[
(\mu \mathcal{S})^n \mu = \frac{1}{2\pi i} \int_{|\lambda|= \rho} \frac{1}{\lambda^{n+2}} \overline{\partial}f^\lambda d\lambda,
\]
multiplying by the characteristic function $\chi_{B_n}$ and by forcing the Jabobian to appear under the integral. This yielded (see \cite[p. 8]{AGRS}) for any $\mu$ with just $\|\mu\|_\infty \leq 1$,  any $M > 1$, and any $E\subset \D$ 
the general estimates
\begin{align}\label{eq:old}
 \|\chi_{E}\psi_{n}\|^2_{L^2}
&\leq  \pi \left( \frac{M+1}{M-1} \right)^{2n} \frac{(M+1)^2}{4}|E|^{1/M}\qquad\textrm{and}
\end{align}
\begin{align}\label{eq:old'}
 \|\chi_{E} \sss(\psi_{n})\|^2_{L^2}
&\leq  \pi \left( \frac{M+1}{M-1} \right)^{2n+2} \frac{(M+1)^2}{4}|E|^{1/M}.
\end{align}
Choosing $E=B_n$ this yields
\begin{align*}
\sqrt{R(n)} 
&\leq  \sqrt{\pi} \left( \frac{M+1}{M-1} \right)^{n} \frac{(M+1)}{2}|B_n|^{1/2M}\\
&\leq \sqrt{\pi} \left( \frac{M+1}{M-1} \right)^{n} \frac{(M+1)}{2}\left( \int_\D e^{pK^\alpha} \right)^{1/2M}e^{-2np/(\beta M)}.
\end{align*}
In our situation we may actually slightly improve this by invoking the  Eremenko and Hamilton form of the area distortion estimate stating for any measurable $E \subset \D$ the inequality
\begin{equation}\label{eq:area_old}
|g(E)| \leq M^{1/M}\pi^{1-1/M}|E|^{1/M} \leq \pi e^{1/(\pi e)} |E|^{1/M}.
\end{equation}
This leads to  
\begin{align*}
\sqrt{R(n)} & \leq \sqrt{\pi} e^{1/(2\pi e)} \left( \frac{M+1}{M-1} \right)^{n} \sqrt{\frac{(M+1)^2}{4M}}\left( \int_\D e^{pK^\alpha} \right)^{1/2M}e^{-2np/(\beta M)}.
\end{align*}

We want to choose $M>1$ in order to force $R(n)$ to decay exponentially. For that we need to have
\[
\log\left( \frac{M+1}{M-1} \right) - \frac{2}{M} \frac{p}{\beta} \leq-\delta < 0.
\]
Choose $M=2p/(p-\beta)$ and estimate
\begin{align*}
\log\left( \frac{M+1}{M-1} \right) - \frac{2}{M} \frac{p}{\beta} &\leq \frac{2}{M-1} - \frac{2}{M} \frac{p}{\beta}
 = \frac{2(p-\beta)}{p+\beta} - \frac{2(p-\beta)}{2\beta}\\
&= -\frac{(p-\beta)^2}{\beta(p+\beta)}
\;\; =:\;-\delta.
\end{align*}
Noting that $\frac{(M+1)^2}{4M}\leq M$ and $ \sqrt{\pi} e^{1/(2\pi e)} \leq 2$ this yields
\begin{equation}\label{eq:R}
\sqrt{R(n)} \leq \widetilde C e^{-\delta n}\qquad \textrm{with}\quad \widetilde C:= \sqrt{\frac{8p}{p-\beta}}\left( \int_\D e^{pK^\alpha}\right)^{(p-\beta)/4p}.\nonumber
\end{equation}
Hence, if we denote $b:=2^{-1/\alpha}\beta^{1/\alpha}$, we obtain
\begin{eqnarray*}
&&\sum_{j=1}^{n} \exp\left(2^{-1/\alpha}\beta^{1/\alpha}\frac{1}{1-1/\alpha}\Big((j+1+\beta/4)^{1-1/\alpha}-(1+\beta/4)^{1-1/\alpha}\Big)\right) \sqrt{R(j)}\\
&\leq& \widetilde Ce^{\widetilde B}\sum_{j=1}^\infty e^{-j\delta/2}\;\leq\; 2\delta^{-1} \widetilde Ce^{\widetilde B},
\end{eqnarray*}
where $\widetilde B:=\sup_{j\geq 1}\Big(\frac{b}{1-1/\alpha}\big((j+1+\beta/4)^{1-1/\alpha}-(1+\beta/4)^{1-1/\alpha}\big)-(\delta/2)j\Big).$ An elementary computation where one simply differentiates with respect to $j$ shows that 
$$
\widetilde B\leq  B:=\max\Big(\frac{ b}{1-1/\alpha}\left((2b/\delta)^{\alpha-1}-(1+\beta/4)^{1-1/\alpha}\right),0\Big).
$$
In view of \eqref{eq:a} we then obtain
\begin{eqnarray*}\|\psi_n-g_n\|_{L^2} &\leq&
  2\delta^{-1} \widetilde Ce^B\exp\left( -2^{-1/\alpha}\beta^{1/\alpha}\frac{1}{1-1/\alpha}\big((n+1+\beta/4)^{1-1/\alpha}-(\beta/4)^{1-1/\alpha}\big)\right).
\end{eqnarray*}
Together with \eqref{eq:aa} this proves the lemma.
\end{proof}

\section{Proof of Lemma \ref{le:kaksi} and Theorem \ref{thm:main2}}\label{sec:proofoflemma2}

We start with an area distortion result that generalizes \cite[Cor. 3.2 and Thm. 5.1]{AGRS} to the range $\alpha \geq1$. In case $\alpha=1$ we need to keep careful track of the constants, which is somewhat non-trivial in this situation, and hence for the readers sake we give the details here although the basic idea of the proof follows that in \cite{Da,AGRS}. Thus for $\alpha=1$ the novelty of the statement as compared to \cite[Thm 5.1]{AGRS} is in a precise estimation of the dependencies of the constant terms. This is a crucial technical ingredient needed for our main results.
\begin{proposition}\label{prop:area} Let $\mu$ and $0<\beta < p$  and $f$ be as in  Lemma \ref{le:Neumann}.

\smallskip

\noindent {\bf (i)}\quad
In case $\alpha>1$  we have the area distortion estimate
\begin{equation}\label{eq:ad1}
|f(E)|\leq c\exp\big(-c'\log^{1-1/\alpha}(e+1/|E|)\big)
\end{equation}
with some constants $c,c'>0$.

\smallskip

\noindent {\bf (ii)}\quad
In case $\alpha=1$, under the additional assumption $1/2<\beta<p<4$, it holds that
\begin{equation}
|f(E)| \;\leq\;  A_2|E|^{\frac{\delta}{24}}+ A_2\delta^{-3\beta}\log^{-\beta}(e
+1/|E|)\left( \int_\D e^{pK}\right)^{1/2}, \qquad E \subset \D,
\end{equation}
where we denoted $\delta:=p-\beta$ and $A_2$ is a universal constant.
\end{proposition}

\begin{proof}[Proof of Proposition \ref{prop:area}]
We start by observing that our maps are  Sobolev homeomorphisms that satisfy Lusin's condition $\mathcal{N}$. Especially, we obtain (using the notation of the previous section) 
\begin{eqnarray}\label{eq:ukk}
|f(E)| &=& \int_E |\partial f|^2 - |\overline{\partial} f|^2 \leq 2|E|+2\int_E |\partial f-1|^2 = 2|E|+2\|\chi_E \big(\partial f-1\big)\|_2^2\\
&\leq & 2|E| +2\Big(\sum_{n=0}^\infty\|\chi_ES\psi_n\|_2\Big)^2.\nonumber
\end{eqnarray}
We estimate the last written sum in two parts, and fix to that end an index $m\geq 1$ that will be specified later on. First of all,  using \eqref{eq:old} with $M=3$ yields
\begin{align}\label{eq:osayks}
\sum_{n=0}^{m-1} \|\chi_ES\psi_n\|_2 \leq \sum_{n=0}^{m-1} \sqrt{\pi} 2^{n+2}  |E|^{1/6}
\leq  2^{m+3} |E|^{1/6}.
\end{align}
In case $\alpha >1$ we  choose $\beta=p/2$ in Lemma \ref{le:Neumann} and obtain with small work the estimate
$$
\sum_{n=m}^{\infty} \|\chi_ES\psi_n\|_2\leq c_1\exp(-c_2m^{1-1/\alpha}).
$$
Here and later $c_j$:s are constants that may depend on $\beta,p,\alpha$, and whose exact value is of no interest to us. By choosing $m= \lfloor 2+\frac{1}{12\log 2}\log 1/|E| \rfloor$ we obtain in view of \eqref{eq:ukk} and the previous estimates
\begin{align*}
|f(E)|\leq 2|E|+c_3|E|^{1/12}+ 4c_1\exp\big(-c_4\log^{1-1/\alpha}(e+1/|E|)\big),
\end{align*}
which proves part (i).

In case (ii) we have $\alpha=1$. In this case we first assume that $2<\beta <p$ and an application of Lemma \ref{le:Neumann} yields in this case
\begin{align*}
\Big(\sum_{n=m}^{\infty} \|\chi_E\psi_n\mu\|_2\Big)^2\leq C_0 \left(\sum_{n=m}^{\infty}  \Big(\frac{n+\beta/4+1}{\beta/4+1}\Big)^{-\beta/2}\right)^2\leq \frac{4C_0(\beta/4+1)^2}{(\beta-2)^2}\Big(\frac{m+\beta/4}{\beta/4+1}\Big)^{2-\beta}
\end{align*}
where the expression for the constant $C_0=C_0(\beta,p)$ is given in \eqref{eq:est2}. In view of \eqref{eq:osayks} we thus have
\[
|f(E)| \leq  2|E|+ 2^{2m+7} |E|^{1/3}+ \frac{8C_0(\beta/4+1)^2}{(\beta-2)^2}\Big(\frac{m+\beta/4}{\beta/4+1}\Big)^{2-\beta} .
\]
Choosing $m=\big\lceil\dfrac{1}{12\log 2}\log(1+1/|E|)\big\rceil$ and noting that $|E|\leq 4|E|^{1/6}$ and $12\log 2\leq 9$  yields
\begin{align}\label{eq:first area}
|f(E)|\leq 2000|E|^{1/6}+ \frac{8C_0(\beta/4+1)^2}{(\beta-2)^2}\Big(\frac{\log((1/|E|+1)^{1/9})+\beta/4}{1+\beta/4}\Big)^{2-\beta} 
\end{align}

Our next step is to apply G. David's factorisation trick to improve the above bound and extend it to all values of $p$. We assume thus  that $f$ is as in the statement of the proposition (with the general assumption $1/2<\beta<p<4$) and
recall from \cite{AGRS} that for any $M\geq 1$ we may factorise $f$   as $f = g \circ F$, where $g$ and $F$ are principal mappings, $g$ is $M$-quasiconformal and $F$ satisfies
\[
I_M:=\int_\D e^{pMK(z,F)} \leq e^M \int_\D e^{pK(z,f)}.
\]
Denote $\beta_0:=( p+\beta)/2$, and $M=2/(\beta_0-\beta)=4/(p-\beta)\geq 1$. 
We will apply \eqref{eq:first area} with parameters $(M\beta_0,Mp)$ instead of $(\beta,p)$ in order to estimate $|F(E)|$. This is possible since by the assumption $p<4$ we have $2<2+\beta M=\beta_0M<pM$. Thus,
\[
|F(E)|\leq 2000|E|^{1/6}+ \frac{8C_0(M\beta_0,Mp,I_M)(M\beta_0/4+1)^2}{(M\beta_0-2)^2}\Big( \frac{\log((1+1/|E|)^{1/9})+M\beta_0/4}{1+M\beta_0/4}\Big)^{2-M\beta_0}. 
\]
Above the notation $C_0(M\beta_0,Mp,I_M)$ recalls the dependences of the constant $C_0$.
 As $g$  is a principal quasiconformal mapping, we obtain from the standard area distortion estimate \eqref{eq:area_old}
$$
|f(E)| = |g \circ F(E)| \leq 4|F(E)|^{1/M} .
$$
By noting that $2-M\beta_0=-M\beta$,  and $M\beta_0/4= \frac{1}{2}\frac{p+\beta}{p-\beta}>(p/\beta-1)^{-1}$, combing the  last two inequalities leads to
\begin{eqnarray}\label{eq:uuu}
&&|f(E)|
\;\leq\; 8000|E|^{1/6M}+\nonumber\\
&&+ 4\cdot8^{1/M}\Big(\frac{C_0(M\beta_0,Mp,I_m)(M\beta_0/4+1)^2}{(M\beta)^2}\Big)^{1/M}\Big( \frac{(\log(1+1/|E|))^{1/6}+(p/\beta-1)^{-1}}{1+(p/\beta-1)^{-1}}\Big)^{-\beta}.
\end{eqnarray} 
Here, since $Mp/M\beta_0-1= (p-\beta)/(p+\beta)\geq (p-\beta)/2p$  and $(p/\beta_0-1)/2M\leq 1/2$ we obtain by recalling \eqref{eq:est2} and easy estimates
\begin{eqnarray*}
&&8^{1/M}\Big(\frac{C_0(M\beta_0,Mp,I_m)(M\beta_0/4+1)^2}{(M\beta)^2}\Big)^{1/M} \; \leq\;
 A_1(p/\beta-1)^{-2\beta}\left( \int_\D e^{pK}\right)^{1/2},
\end{eqnarray*} 
where $A_1$ is an absolute constant. In the simplification  we applied our assumption on the range of $p$ and $\beta$ and observed that $(p-\beta)^{-(p-\beta)}$ has a universal upper bound. We  also observe in \eqref{eq:uuu} that $(p/\beta-1)^\beta$ has a universal upper bound, and by increasing $A_1$ we may replace $\log(1+1/|E|)^{1/6}$ by $\log(1+1/|E|)$. In addition, in our situation $p/\beta-1\asymp (p-\beta)$.  Combining these estimates
completes the proof of part (ii).
\end{proof}

We then turn to the goals stated in the title of this section. As expected, we will first estimate the integrability of the Jacobian using the estimates for area-distortion we just proved.  For that purpose we will first state a general lemma that yields (essentially optimal) integrability estimates from estimates of area distortion.

\begin{lemma}\label{le:HLP} Assume that $f$ is a principal mapping of finite distortion and $g:[0,\pi)\to [0,\infty)$ is concave with $g(0)=0$, satisfying for any measurable subsets $E\subset B(0,1)$ the area distortion estimate
\begin{equation}\label{eq:ad}
|f(E)|\leq g(|E|).
\end{equation}
Then for any convex and increasing  $H$ on $[0,\infty)$ it holds that
$$
\int_{B(0,1)}H(J_f(z))dA(z)\leq \int_0^{\pi}H(g'(t))dt.
$$
\end{lemma}
\begin{proof} Let us denote by $h:(0,\pi)\to \R_+$ the decreasing rearrangement of $J_f$. By assuming first that $g$ is differentiable on $(0,\pi)$, our assumption  may be rewritten as 
$$
\int_0^xh(t)dt\leq \int_0^xg'(t)dt\qquad\textrm{for all} \;\,x\in (0,\pi).
$$
The statement now follows from a continuous version of the Hardy-Littlewood P\'olya (or Karamata) inequality, see  \cite[Theorem 2.1]{C} or \cite{HLP}.
\end{proof}
\begin{proof}[Proof of Theorem  \ref{thm:main2}]  It follows from Lemma \ref{le:HLP} and Proposition \ref{prop:area}(i) that in our situation the higher integrability of $J_f$ is at least as good as that of the derivative $h'$  on the interval $(0,\pi)$ where
$$
h(x):=\exp\big(-c'\log^{1-1/\alpha}(e+1/x)\big) $$
 on the interval $(0,\pi)$. Namely, $h$ is clearly decreasing near the origin which is enough for us in order to apply Lemma \ref{le:HLP}. We may safely leave to the reader to check that $\phi(h')$ is integrable near the origin with $\phi(y):=y\exp(\log^\beta (e+y))$ for $\beta<1-1/\alpha.$ In other words, we have
 \begin{equation*}\label{eq:alpha_jacob}
 \int_\D J_f\exp(\log^\beta (e+J_f))<\infty \qquad\textrm{for}\quad \beta<1-\alpha^{-1}.
 \end{equation*}
 By recalling that $|Df|^2=KJ_f$, the stated integrability of the derivative follows immediately by the elementary inequality
 $$
 xy\exp\big(\log^{\beta'} (e+xy)\big)\leq C\Big(\exp(px^{\alpha})+ y\exp\big(\log^\beta (e+y)\big)\Big), \qquad x,y\geq 1.
 $$
 for any $0<\beta'<\beta<1$ and $p>0$, and where $C=C(p,\beta,\beta',\alpha).$ The latter inequality follows easily by examining separately the cases $x<\exp((1/2)\log^\beta(e+y))$ and $x\geq\exp((1/2)\log^\beta(e+y))$.
\end{proof}

\begin{proof}[Proof of Lemma \ref{le:kaksi}]

Easy estimates that just apply differentiation show that the function 
$$
x\mapsto (1+\delta\log(1+1/x))^{-\beta}
$$
is concave for $x>0$ as soon as $\delta <(1+\beta)^{-1}$, which in our situation holds at least if $\delta<1/5.$ We now fix $p=1+2\varepsilon$, $\beta=1+\varepsilon$, with $\varepsilon\in (0,1/10)$ in Proposition \ref{prop:area} (ii)  and note   that  Lemma \ref{le:HLP} yields the integrability
\begin{eqnarray*}
\int_\D J(z,f) \log(e + J(z,f)) \leq \int_0^\pi h'(x)\log(e + h'(x))dx, 
\end{eqnarray*}
where $h(x):= A_2x^{\frac{\varepsilon}{24}}+ A_2\varepsilon^{-3-3\varepsilon}\Big(\log(1+1/x)\Big)^{-1-\varepsilon}\left( \int_\D e^{pK}\right)^{1/2}$ . Hence, if we   denote $A_3:=A'_2\left( \int_\D e^{pK}\right)^{1/2}$ with another universal constant  $A'_2$ we have
\begin{eqnarray*}
&h'(x)\le &A_3 \bigg(\varepsilon x^{-1+\frac{\varepsilon}{24}}+\frac{\varepsilon^{-3}}{x}\Big(\log(1+1/x)\Big)^{-2-\varepsilon} \bigg),
\end{eqnarray*}
Obviously $\log h'(x)\leq \log (A_3)+3\log(1/\varepsilon)+ 3\log(10/ x),$
so that noting that $\int_0^\pi h'(x)dx =h(\pi)\leq 10A_3\varepsilon^{-3}$ we obtain that
\begin{eqnarray*}
\int_\D J(z,f) \log(e + J(z,f)) &\leq &\int_0^\pi h'(x)\log(e + h'(x))dx\\
&\leq &10A_3\varepsilon^{-3}\big( \log (A_3)+\log(1/\varepsilon)\big)\\
& & \,+\,3A_3\int_0^\pi \bigg(\varepsilon x^{-1+\frac{\varepsilon}{24}}+\frac{\varepsilon^{-3}}{x}\Big(\log(1+1/x)\Big)^{-2-\varepsilon} \bigg)\log(10/x)dx\\
&\leq& A_3\log(A_3)10^6\varepsilon^{-4}.
\end{eqnarray*}
In the estimation of the last written integral we noted that $\int_0^\pi\varepsilon x^{-1+\frac{\varepsilon}{100}}\log(10/x)dx\leq 10^5\varepsilon^{-1}$ and estimated the second  integral from the above by
$
2\log (20)\Big(\varepsilon^{-3}\int_0^{1/2}\log(1/x)^{-1-\varepsilon}\frac{dx}{x}+ 3\varepsilon^{-3} )\leq 40\varepsilon^{-4}.
$

We next note the well-known inequality stating that for any $\varepsilon\in (0,1)$ and reals $x,y>0$ it holds that
\[
xy  \leq  x \log(e+x) + e^{(1+\varepsilon) y}
\]
(one simply checks that is true  for $\varepsilon=0$). The choice  $x=J_j(z)$, $y=K:=K_f(z)$, and integration over $\D$ finally yields that
\begin{eqnarray}\label{eq:crucial}
\int_\D |Df|^2 &\leq& A_4 \left( \int_\D e^{(1+\varepsilon)K}\right)^{1/2}\log \left( \int_\D e^{(1+\varepsilon)K}\right) \varepsilon ^{-4} + \int_\D e^{(1+\varepsilon)K} 
\leq
 A_5 \varepsilon ^{-4}\int_\D e^{(1+\varepsilon)K},
\end{eqnarray}
where $A_4,A_5$ are universal constants.

We are now ready to complete the proof of Lemma \ref{le:kaksi}. To that end  we need to establish for any $w$ with $\Re w = 1$ the key estimate
\[
\int_{\D} | g_w|^2 \leq \frac{C}{\varepsilon^4},
\]
with  constant $C$ does not depend on $\varepsilon$. Note that
this estimate implies the bound $M_1 \leq \frac{C_1}{\varepsilon^{2}}$ for some constant $C_1$. Moreover, estimating the integrability of  $|g_w|$ reduces to that of $|\partial(f_w)|$ because $|g_w|=|\partial(f_w)|$ a.e. since we have  $\Re w = 1.$

Let us first estimate the distortion $K(z,f_w)$. 
Assume that $\varepsilon\in (0,1/2)$ and consider the function $r(x):=1-x^{1+\varepsilon}-(1+\varepsilon/2)(1-x).$ We claim that $r(x)\geq 0$ for $x\in [1/2,1]$. As $r$ is concave with $r(1)=0$ and $r'(1)=-\varepsilon/2 <0$, it is enough to check that $r(1/2)\geq 0, $ or equivalently that
$1+\varepsilon/2\leq 2-2^{-\varepsilon}$. In turn this follows from the concavity of $\varepsilon\mapsto R(\varepsilon):=2-2^{-\varepsilon}-(1+\varepsilon/2)$ and by noting that $R(0)=R(1/2)=0$.

We thus have that $1-|\nu(z)|^{1+\varepsilon}\geq (1+\varepsilon/2)(1-|\nu(z)|)$ assuming that  $|\nu(z)|\geq 1/2$, and  we may estimate the distortion as follows:
\begin{align*}
K(z,f_w) &= \frac{1+|\nu(z)|^{1+\varepsilon}}{1-|\nu(z)|^{1+\varepsilon}} = -1 +  \frac{2}{1-|\nu(z)|^{1+\varepsilon}} \leq -1 +  \frac{2}{\min((1+\varepsilon/2)(1-|\nu(z)|),1-1/2^{1+\varepsilon})}\\
& \leq -1 +  \frac{2}{(1+\varepsilon/2)(1-|\nu(z)|)} + \frac{2}{1-1/2^{1+\varepsilon}} \leq 3 + \frac{2}{(1+\varepsilon/2)(1-|\nu(z)|)} \\
&= \frac{1}{1+\varepsilon/2} \left( -1 + \frac{2}{1-|\nu(z)|} \right) + 3 + \frac{1}{1+\varepsilon/2} \leq \frac{K(z,f)}{1+\varepsilon/2} + 4.
\end{align*}
It follows that $\int_{\D} e^{(1+\varepsilon/2)K(z,f_w)} \leq \int_{\D} e^{K(z,f)+4+2\varepsilon} \leq e^{5}\int_{\D} e^{K(z,f)}$.

In conclusion, an application of inequality \eqref{eq:crucial} (with $\varepsilon/4$ in place of $\varepsilon$) yields the desired uniform bound
\[
\int_{\D} |\partial(f_w)|^2 \leq \int_{\D} |Df_w|^2 \leq \frac{C}{\varepsilon^4}.
\]
\end{proof}

\section{Proof of Theorem \ref{th:radial}}\label{se:radialproof}

 Throughout this section we assume that $f:\D\to\D$ is a radial homeomorphism of the form 
$$
f(z)=\frac{z}{|z|}\phi(z),
$$
where $\phi:[0,1]\to[0,1]$ is an increasing homeomorphism. We also assume that $f$ is a map with exponentially integrable distortion (satisfying \eqref{eq:fed}). By the Lusin condition this implies that $\phi$ is absolutely continuous, and the assumed exponential integrability of $K_f$ can expressed as
\begin{eqnarray}\label{eq:ehto2}
\int_0^1\Big(e^{p\frac{r\phi'(r)}{\phi(r)}}+e^{p\frac{\phi(r)}{r\phi(r)}}\Big)rdr =C_0<\infty.
\end{eqnarray}
Our aim is to first prove an area distortion estimate for these maps.
\begin{proposition}\label{pr:area3}
Let   $f:\D\to\D$ be a radial homeomorphism of exponentially integrable distortion $($see \eqref{eq:fed}$)$.  Then for any measurable subset $E\subset \D$.
\begin{equation}\label{eq:area3}
|f(E)|\leq C\big(\log(1+1/|E|)\big)^{-p}.\nonumber
\end{equation}
The constant $C=C(p,C_0)$ is uniform for fixed $C_0$ and $p\in[1,2].$
\end{proposition}
\begin{proof}
We shall denote by $C$ constants whose actual size if of no interest to us, and their value may change from line to line. We call the  set $E\subset \D$ `radial' if one has that  $z\in E$ if and only if $ |z|\in E$.
By a standard approximation argument it is enough to prove the claim in the case where $E$ is a disjoint union of sets of the form $\{a<|z|<b,\; \alpha_0<{\rm arg}(z)<\alpha_1\}$, and this case is easily reduced to the case of radial sets. Thus, we may assume that $E=\{|z|\in F,\}$ where $F\subset (0,1)$ is a disjoint union of open intervals.

For $n=1,2,\ldots$  we denote the dyadic annuli $A_n:=\{2^{-n}\leq |z|\leq 2^{1-n}\}.$ Our first goal is to estimate $\phi(r)$ from the above. To that end, fix $n\geq 1$ and note that by \eqref{eq:ehto2} and Jensen's inequality applied on the probability  measure $r^{-1}dr$ on $(e^{-n}, e^{1-n})$ and on the convex function $x\mapsto e^{p/x}$ yields 
\begin{eqnarray*}
\log\big(\phi(e^{1-n})/\phi(e^{-n})\big) &=&\int_{e^{-n}}^{e^{1-n}}\frac{r\phi'(r)}{\phi(r)}\frac{dr}{r}\geq p\left(\log\left(\int_{e^{-n}}^{e^{1-n}}\exp\bigg(p\frac{\phi(r)}{r\phi'(r)}\bigg)\frac{dr}{r}\right)\right)^{-1}\\
&\geq&p\left(\log\left(e^{2n}\int_{e^{-n}}^{e^{1-n}}\exp\bigg(p\frac{\phi(r)}{r\phi'(r)}\bigg)rdr\right)\right)^{-1}\\
&\geq & \frac{p}{\log C_0+2n}.
\end{eqnarray*}
Applying this for first $n$ annuli yields 
\begin{equation}\label{eq:obtain1}
\phi(e^{-n})\leq \exp \left(\sum_{k=1}^n\frac{p}{\log C_0+2n}\right)\leq Cn^{-p/2}.
\end{equation}

We next produce a very crude estimate of  area distortion for radial sets $E\subset A_n$. Write $E=\{|z|\in F,\}$ where $F\subset (e^{-n},e^{1-n})$ and note that $|F|\leq e^n|E|.$ Let us observe first that \eqref{eq:ehto2}, Jensen's inequality and the convexity of the map $x\mapsto\exp(p\sqrt{x+1})$ on $[0,\infty)$ yield that
\begin{eqnarray*}
\int_{e^{-n}}^{e^{1-n}}\Big(\frac{r\phi'(r)}{\phi(r)}\Big)^2\frac{dr}{r}&\leq& \left(\frac{1}{p}\log\left(\int_{e^{-n}}^{e^{1-n}}\exp\bigg(p\sqrt{\Big(\frac{r\phi'(r)}{\phi(r)}\Big)^2+1}\bigg)\frac{dr}{r}\right)\right)^2-1\\
&\leq&\left(\frac{1}{p}\log\left(e^{2n}\int_{e^{-n}}^{e^{1-n}}\exp\bigg(p\frac{r\phi'(r)}{\phi(r)}+p\bigg)rdr\right)\right)^{2}-1\\
&\leq & p^{-2}(\log C_0+2n+p)^2\;\leq \; Cn^2.
\end{eqnarray*}
We may then compute using the above estimate, the bound \eqref{eq:obtain1} and Cauchy-Schwarz to obtain for radial subsets of $E\subset A_n$
\begin{eqnarray}\label{eq:obtain2}
|f(E)|&=&2\pi \int_{F}\phi(r)\phi'(r)dr\leq 2\pi (\phi(e^{n+1}))^2\int_{F}\frac{r\phi'(r)}{\phi(r)}\frac{dr}{r}\nonumber\\
&\leq& \frac{C}{n^p} \sqrt{\int_{F}\frac{dr}{r}}\sqrt{\int_{e^{-n}}^{e^{1-n}}\Big(\frac{r\phi'(r)}{\phi(r)}\Big)^2\frac{dr}{r}}\nonumber\\
&\leq& \frac{C}{n^p}  \sqrt{|F|}e^{n/2}n\; \leq Cn^{1-p}e^n\sqrt{E}.
\end{eqnarray}
We finally observe that in the general case we may assume that $|E|=e^{-4N}$ for some integer $N\geq 1$. By using  the estimates \eqref{eq:obtain1} and  \eqref{eq:obtain2} it follows that
\begin{eqnarray*}
|f(E)|&\leq& |f(\{ |z|\leq e^{-N}\})|+ \sum_{n=1}^{N}|f(E\cap A_n)|\leq \pi(\phi(e^{-N}))^2+C\sum_{n=1}^N
e^nn^{1-p}\sqrt{e^{-4N}}\\
& \leq & \frac{C}{N^p}+ Ne^{-N}
 \leq  \frac{C'}{4N^p},
\end{eqnarray*}
as was to be shown.
\end{proof}

\begin{proof}[Proof of Theorem  \ref{th:radial}] One simply applies the area distortion estimate we just proved and obtains the analogue   of \eqref{eq:crucial} now with term $1/\varepsilon$ instead of $1/\varepsilon^4$. The first part follows then directly from Lemma \ref{le:apu3}. Similarly, part (ii) follows by keeping the track of the dependence of constant factors under this area distortion estimate.

\end{proof}

\bibliographystyle{amsplain}

\end{document}